\documentclass[]{amsart}

\usepackage{graphicx} 
\usepackage{tabularx}
\usepackage[margin=4cm]{geometry}

\usepackage{lmodern}
\usepackage[T1]{fontenc}
\usepackage[utf8]{inputenc}

\usepackage{tocloft}
\usepackage{amssymb, amsmath, amsthm, amsfonts}
\usepackage{physics}
\usepackage{mathrsfs}
\usepackage{upgreek}
\usepackage{hyperref}

\usepackage{fancyhdr}
\usepackage{titlesec}

\titleformat{\section}[block]
  {\centering\bfseries\Large}
  {\thesection.}{1em}{}

\titleformat{\subsection}[block]
  {\centering\bfseries\normalsize}
  {\thesubsection}{1em}{}

\titleformat{\subsubsection}[runin]
  {\bfseries}
  {\thesubsubsection}{1em}{}

\titlespacing*{\section}{0pt}{1.5ex plus 0.5ex minus 0.2ex}{1ex}

\newtheorem{theorem}{Theorem}

\newtheorem{lemma}{Lemma}

\newtheorem{remark}{Remark}
\newtheorem{proposition}{Proposition}

\newcommand{\C}{\mathbb{C}}

\newcommand{\N}{\mathbb{N}}

\newcommand{\calh}{\mathcal{H}}
\newcommand{\calk}{\mathcal{K}}
\newcommand{\ball}{\text{Ball}}

\newcommand{\cala}{\mathcal{A}}
\newcommand{\cald}{\mathcal{D}}

\title{A Separable Hilbertian Operator Space with CBAP and No Completely Bounded Basis}
\author{Dominique Guillot \and Nicolas Cutrona}
\date{\today}

\keywords{operator spaces; completely bounded approximation property; completely bounded bases; completely bounded frames; Hilbertian operator spaces; completely bounded maps; $\Pi_2$-hyperreflexivity.}
\subjclass[2020]{46L07 (primary), 46B28, 47L25 (secondary)}

\begin{document}

\begin{abstract}
Liu and Ruan showed that a separable operator space has the completely bounded approximation property if and only if it admits a completely bounded frame. We show that, even for Hilbertian operator spaces, such a frame need not give rise to a completely bounded basis. More precisely, we construct a separable Hilbertian operator space with the $1$-CBAP but with no completely bounded basis. Our construction relies on a theorem of Oikhberg that provides a structural description of the completely bounded maps on a suitable operator space.
\end{abstract}

\maketitle

\section{Introduction}
Approximation properties have played a fundamental role in the structure
theory of Banach spaces, where one studies to what extent the identity map can
be approximated by finite-rank operators.  In the operator-space setting, the
analogous questions are formulated using completely bounded maps. Recall that an \emph{operator space} is a Banach space $X$ equipped with a
compatible sequence of matrix norms on $M_n(X)$, $n\geq 1$, satisfying Ruan's axioms:
\begin{enumerate}
    \item[(R1)] If $x\in M_n(X)$ and $\alpha,\beta\in M_n$, then
    \[
    \|\alpha x\beta\|_{M_n(X)}
    \leq
    \|\alpha\|_{M_n}\,\|x\|_{M_n(X)}\,\|\beta\|_{M_n}.
    \]

    \item[(R2)] If $x\in M_n(X)$ and $y\in M_m(X)$, then
    \[
    \|x\oplus y\|_{M_{n+m}(X)}
    =
    \max\{\|x\|_{M_n(X)},\|y\|_{M_m(X)}\}.
    \]
\end{enumerate}
If $X$ and $Y$
are operator spaces and $T:X\to Y$ is linear, its $n$th amplification is
defined by
\[
    T^{(n)}:M_n(X)\longrightarrow M_n(Y),
    \qquad
    T^{(n)}([x_{i,j}])=[T(x_{i,j})].
\]
We say that $T$ is \emph{completely bounded} if
\[
    \|T\|_{\mathrm{cb}}
    :=
    \sup_{n\geq1}
    \|T^{(n)}\|_{B(M_n(X),M_n(Y))}
    <\infty,
\]
where $B(X,Y)$ denotes the space of bounded linear operators from Banach spaces $X$ to $Y$. See e.g.~\cite{EffrosRuan1990, EffrosRuan2000, Paulsen2002, Pisier2003} for more details about operator spaces.

An operator space $X$ is said to have the \emph{completely bounded
approximation property} (CBAP) if there exists a net of finite-rank operators
$T_\alpha:X\to X$ such that
\[
    T_\alpha x\longrightarrow x
    \qquad (x\in X)
\]
and
\[
    \sup_\alpha\|T_\alpha\|_{\mathrm{cb}}<\infty.
\]
If $\sup_{\alpha}\norm{T_\alpha}_{\mathrm{cb}} \le C$, we say that $X$ has the $C$-CBAP. In particular, if $C \le 1$, we say that $X$ has the $1$-CBAP or the \textit{completely contractive approximation property}. Thus, the CBAP is the natural operator-space analogue of the bounded
approximation property for Banach spaces.

A more rigid form of approximation is obtained from a basis.  The notion of
a completely bounded (cb) basis was introduced and studied by Junge, Nielsen,
Ruan, and Xu \cite{JungeNielsenRuanXu2004}.  A separable operator space $X$
is said to have a \emph{cb-basis} if $X$ has a Schauder basis
$\{e_n\}_{n\geq1}$ whose natural partial-sum projections
\[
    P_m\left(\sum_{n\geq1}\alpha_ne_n\right)
    =
    \sum_{n=1}^m\alpha_ne_n
\]
satisfy
\[
    \sup_{m\geq1}\|P_m\|_{\mathrm{cb}}<\infty.
\]
Since each $P_m$ has finite rank and $P_mx\to x$ for every $x\in X$, every
operator space with a cb-basis has the CBAP.  Such cb-bases have been shown
to exist in several important classes of operator spaces.  In particular,
Junge, Nielsen, Ruan, and Xu proved that every separable nuclear
$C^*$-algebra admits a cb-basis.

Frames provide a less rigid form of reconstruction.  In Banach-space theory,
Schauder frames allow expansions of the form
\[
    x=\sum_{n\geq1}f_n(x)x_n, \qquad x_n \in X, \: f_n \in X^*,
\]
without requiring the sequence $\{x_n\}_{n \ge 1}$ to form a Schauder basis.  Such
redundancy makes frames particularly well suited to approximation-theoretic
questions; see, for example, Casazza, Han, and Larson
\cite{CasazzaHanLarson1999}.

Motivated by this classical framework, Liu and Ruan
\cite{LiuRuan2016} introduced the notion of a cb-frame for operator
spaces.  A sequence
\[
    \{(x_n,f_n)\}_{n\geq1}\subset X\times X^*
\]
is a \emph{cb-frame} if
\[
    x=\sum_{n\geq1}f_n(x)x_n
    \qquad (x\in X)
\]
and the associated partial reconstruction operators
\[
    S_mx:=\sum_{n=1}^m f_n(x)x_n
\]
satisfy
\[
    \sup_{m\geq1}\|S_m\|_{\mathrm{cb}}<\infty.
\]

A fundamental result of Liu and Ruan shows that, for separable operator
spaces, cb-frames characterize the CBAP:
\[
    X\text{ has the CBAP}
    \quad\Longleftrightarrow\quad
    X\text{ admits a cb-frame}.
\]
In fact, they prove the stronger equivalence that these conditions hold if
and only if $X$ is completely isomorphic to a completely complemented
subspace of an operator space with a cb-basis.  Their result identifies
cb-frames as the natural reconstruction-theoretic manifestation of the
CBAP.

Following Liu and Ruan's work, it is natural to ask whether, under additional geometric assumptions, the existence of a cb-frame forces the existence of a cb-basis.
For arbitrary operator spaces, the answer is negative for reasons already
visible in Banach-space theory.  Indeed, as shown by Szarek, there exist separable Banach spaces
with the bounded approximation property but without a Schauder basis \cite{Szarek1987}.  If
$E$ is such a Banach space and $\operatorname{MIN}(E)$ is its minimal
operator-space structure, then bounded maps into $\operatorname{MIN}(E)$ are
automatically completely bounded with the same norm.  Hence
$\operatorname{MIN}(E)$ has the CBAP whenever $E$ has the bounded
approximation property.  Since the underlying Banach space $E$ has no
Schauder basis, $\operatorname{MIN}(E)$ cannot admit a cb-basis.  Thus,
in general, the CBAP does not imply the existence of a cb-basis.

The above barrier disappears for Hilbert spaces.  Every separable
Hilbert space has a countable orthonormal basis, and hence a Schauder basis.  It is
therefore natural to ask whether the additional Hilbert-space geometry
eliminates the obstruction at the completely bounded level.  More precisely,
if $X$ is a separable \emph{Hilbertian operator space}, meaning that its underlying
Banach space is isometric to a Hilbert space, it is natural to ask whether the CBAP
forces $X$ to admit a cb-basis. The main result of this paper shows that this is not the case.

\begin{theorem}\label{mainresult}
There exists a separable operator space $X$ whose underlying Banach space is
isometric to an infinite-dimensional Hilbert space such that:
\begin{enumerate}
    \item $X$ has the $1$-CBAP;
    \item $X$ does not have a cb-basis.
\end{enumerate}
\end{theorem}

Our construction is based on work of Oikhberg
\cite{Oikhberg2005, Oikhberg2007, Oikhberg2009} and Oikhberg–Ricard \cite{OikhbergRicard2004} which produces operator spaces with structured completely bounded maps.  We apply Oikhberg's theorem \cite[Theorem~3.1]{Oikhberg2009} to a particular
block-diagonal algebra on a separable Hilbert space. As we show, Oikhberg's construction yields an operator space having the $1$-CBAP whose completely
bounded maps structure forces every uniformly cb-family of finite-rank idempotents
to remain uniformly close, in Hilbert--Schmidt norm, to a family of block projections.  By choosing the block dimensions with sufficiently large
gaps, we prove that such projections cannot have every possible finite rank.
This rules out the existence of a cb-basis.

\section{The Construction}

Before proving Theorem \ref{mainresult}, we recall some tools that we will need along the way.

\subsection{Hilbert-Schmidt Class and Absolutely 2-summing Operators}

Let $E,F$ be Banach spaces and let $B(E,F)$ be the set of bounded linear maps from $E$ to $F$. Recall that $T \in B(E,F)$ is \textit{absolutely 2-summing} if there exists a constant $C$ such that for every finite collection $x_1,\ldots, x_m \in E$ 
\begin{equation}\label{2sum}
        \left(\sum^m_{i=1}\norm{Tx_i}^2_{F}\right)^{1/2} \le C\sup_{\varphi \in \ball(E^*)}\left(\sum^m_{i=1}|\varphi(x_i)|^2\right)^{1/2}.
\end{equation}
We denote the space of absolutely 2-summing operators by $\Pi_2(E,F)$, equipped with a norm $\pi_2(\cdot)$ defined to be the smallest constant $C$ such that (\ref{2sum}) holds. 

Let $\calh, \calk$ be Hilbert spaces and let $\{e_i\}_{i \in I}$ be an orthonormal basis for $\calh$. The \textit{Hilbert-Schmidt class} is the family of operators $S_2(\calh, \calk) = \{T \in B(\calh, \calk): \sum_{i \in I}\norm{Te_i}_\calk^2 < \infty\}$ with Hilbert-Schmidt norm $\norm{T}_{2} = (\sum_{i \in I}\norm{Te_i}_\calk^2)^{1/2}$. Recall that, for Hilbert spaces, the absolutely $2$-summing operators coincide with the Hilbert--Schmidt class. This will allow us, when discussing $\Pi_2$-hyperreflexivity, to work with the Hilbert-Schmidt norm. 

\begin{proposition}[{\cite[Chapter~4]{DiestelJarchowTonge1995}}]\label{prop2sumpi}
    If $\calh, \calk$ are Hilbert spaces, then
    \[
    \Pi_2(\calh, \calk) = S_2(\calh, \calk),
    \]
    and 
    \[
        \pi_2(T) = \norm{T}_2.
    \]
\end{proposition}

\subsection{Oikhberg's Theorem}

Let $X$ be a Banach space and let $\mathcal A\subseteq B(X)$ be a linear subspace. For $T\in B(X)$, define
\[
    \alpha_{\mathcal A}(T)
    :=
    \sup \left\{
        \pi_2(qTi):
        \begin{array}{l}
        i:E\to X \text{ is an isometric embedding},\\
        q:X\to F \text{ is a quotient map},\\
        E,F \text{ are finite-dimensional},\\
        qAi=0 \text{ for every } A\in\mathcal A
        \end{array}
    \right\}.
\]
We say that $\mathcal A$ is $C$-$\Pi_2$-hyperreflexive \cite{Oikhberg2007, Oikhberg2009} if
\[
    \inf_{A\in\mathcal A}\pi_2(T-A)
    \leq
    C\,\alpha_{\mathcal A}(T)
\]
for every $T\in B(X)$ for which
$\alpha_{\mathcal A}(T)<\infty$.

We  now state Oikhberg's Theorem.

\begin{theorem}[{see \cite[Theorem~3.1, Proposition~2.1]{Oikhberg2009}}]\label{oikgeneral}
Suppose $E$ is a Banach space with separable dual, and $\mathcal{A} \subset B(E)$ is a subspace such that $\normalfont \ball(\cala)$ is a balanced convex subset of $\normalfont \ball(B(E))$, closed in the topology
induced by $E\widehat{\otimes}_\pi E^*$, containing the identity, and such
that $ab\in\mathcal{A}$ whenever $a,b\in\mathcal{A}$.

Suppose, furthermore, that $\mathcal{A}$ is
$C$--$\Pi_2$-hyperreflexive. Then there exists an operator space $X$,
isometric to $E$, such that
\[
    CB(X)=\mathbb{C}\mathcal{A}+\Pi_2(E).
\]
More precisely:
\begin{enumerate}
    \item If $a\in \ball(\mathcal{A})$, then
    \[
        \|a\|_{\mathrm{cb}}\leq 1.
    \]
    \item If $S\in\Pi_2(X)$, then
    \[
        \|S\|_{\mathrm{cb}}\leq \pi_2(S).
    \]

    \item If $T\in CB(X)$, then there exist $S\in\Pi_2(X)$,
    $a\in\ball(\mathcal{A})$, and $\lambda\in\mathbb{C}$ such that
    \[
        T=\lambda a+S,
    \]
    with
    \[
        |\lambda|
        \leq
        4(C+1)\|T\|_{\mathrm{cb}}
    \]
    and
    \[
        \pi_2(S)
        \leq
        4C\|T\|_{\mathrm{cb}}.
    \]
\end{enumerate}
\end{theorem}

\begin{remark}
    \normalfont Theorem \ref{oikgeneral} is a restatement of \cite[Theorem~3.1]{Oikhberg2009}, where $\Pi_2$-Azoff--Shehada hyperreflexivity of the unit ball is replaced with $\Pi_2$-hyperreflexivity, as \cite[Proposition~2.1]{Oikhberg2009} provides the passage from the
$\Pi_2$-hyperreflexivity condition above to that hypothesis.
\end{remark}

The notion of $\Pi_2$-hyperreflexivity simplifies when considering Hilbert spaces. Let $\cala \subseteq B(\calh)$ be a linear subspace, where $\calh$ is a Hilbert space. Using Proposition \ref{prop2sumpi}, for $T \in B(\calh)$ define
\[
\alpha_\cala(T):= \sup \{\norm{u Tv}_2: u,v\text{ finite-rank contractions} \in B(\calh),\: uAv = 0 \text{ for every }A \in \cala\}.
\]
Informally, the quantity $\alpha_\cala(T)$ measures the pieces of $T$ which can be detected in positions where every operator in $\cala$ vanishes. Then $\cala$ is \textit{$\Pi_2$-hyperreflexive} if there exists $C > 0$ such that
\begin{equation}
    \inf_{A \in \cala} \norm{T-A}_{2} \le C\alpha_\cala(T),
\end{equation}
whenever the right-hand side is finite. In this case, we say that $\cala$ is $C$-$\Pi_2$-hyperreflexive. This gives a measure of ``closeness'' of $T$ and $\cala$ in Hilbert-Schmidt norm. 

For our construction, we formulate a special case of Oikhberg's Theorem. Let $\calh$ be a separable Hilbert space. Since $\calh \cong \calh^*$, $\calh^*$ is separable. If we let $\cala \subset B(\calh)$ be a unital $\Pi_2$-hyperreflexive operator algebra, then $\ball(\cala)$ is both convex and balanced. Moreover, suppose that $\ball(\cala)$ is weak$^*$ closed. On $B(\mathcal H)$, the topology induced by $\mathcal H\widehat{\otimes}_\pi\mathcal H^*$ is the weak$^*$ topology $\sigma(B(\mathcal H),S_1(\mathcal H))$. Hence $\operatorname{Ball}(\mathcal A)$ is closed in the topology induced by $\mathcal H\widehat{\otimes}_\pi\mathcal H^*$.
 These facts lead to the following formulation of Theorem \ref{oikgeneral} for Hilbert spaces.

\begin{theorem}[Oikhberg, Hilbert space version]\label{thm:oikhberg-specialized}
Let $\calh$ be a separable Hilbert space and let
$\mathcal A\subset B(\calh)$ be a unital $\Pi_2$-hyperreflexive operator
algebra whose closed unit ball is weak$^*$ closed.

Then there exists an operator-space structure $X$ on $\calh$
such that:
\begin{enumerate}
    \item the level-one norm on $X$ is exactly the original Hilbert-space
    norm;
    \item every $A\in\mathcal A$ with $\|A\|_{B(\calh)}\leq1$ satisfies
    \[
        \|A\|_{\mathrm{cb}}\leq1;
    \]
    \item there are constants $C_0,C_1>0$ such that every
    $T\in CB(X)$ can be written
    \[
        T=A+S,
        \qquad
        A\in\mathcal A,\quad S\in S_2(\calh),
    \]
    with
    \[
        \|A\|_{B(\calh)}\leq C_0\|T\|_{\mathrm{cb}},
        \qquad
        \|S\|_2\leq C_1\|T\|_{\mathrm{cb}}.
    \]
\end{enumerate}
\end{theorem}

\begin{remark}
    \normalfont For our construction, the main content of Theorem \ref{thm:oikhberg-specialized} is $CB(X) \subset \cala + S_2(\calh)$ with quantitative control, while all contractions in $\cala$ become complete contractions.
\end{remark}

\subsection{A Block Algebra}

For each $k \in \N$, let $\calh_k$ be a Hilbert space of dimension $d_k:=4^k$. Let
\[
 \calh= \bigoplus_{k \geq 1} \calh_k,
\]
equipped with the inner product
\[
\langle \zeta, \eta\rangle_\calh = \sum_{k\ge1}\langle \zeta_k, \eta_k\rangle_{\calh_k}, \qquad \zeta = (\zeta_k)_{k=1}^\infty, \eta = (\eta_k)_{k=1}^\infty.
\]
Let $E_k \in B(\calh)$ denote the orthogonal projection onto $\calh_k$. Denoting the identity operator on $\calh_k$ by $I_{\calh_k}$, we define a block algebra
\[
\cald := \left\{  \bigoplus_{k \geq 1} \lambda_kI_{\calh_k}: \{\lambda_k\}_{k\ge 1} \in \ell_\infty \right\}. 
\]
Thus, an operator $D \in \cald$ has the form 
\[
    D=
    \begin{pmatrix}
    \lambda_1 I_{\calh_1} & 0 &0&\cdots\\
    0&\lambda_2 I_{\calh_2}&0&\cdots\\
    0&0&\lambda_3I_{\calh_3}&\cdots\\
    \vdots&\vdots&\vdots&\ddots
    \end{pmatrix}.
\]

We make a few observations about $\cald$. First, it is clear that $\cald$ is a subspace of $B(\calh)$. Even more, $\cald$ is an algebra. Indeed, if $D_1 = \oplus_k\lambda_kI_{\calh_k}$ and $D_2 = \oplus_k \mu_kI_{\calh_k}$, then $D_1D_2 = \oplus_k\lambda_k\mu_kI_{\calh_k} \in \cald$ since $\sup_{k}|\lambda_k\mu_k|$ is finite. 

Next, if $D \in \cald$ is idempotent, then we must have
\[
\lambda_k = \lambda_k^2, \quad k \in \N,
\]
so that $\lambda_k \in \{0,1\}$. In particular, a finite-rank projection in $\cald$ has rank
\[
\sum_{k \in F} 4^k,
\]
for some finite subset $F$ of $\N$. Let
\begin{equation}\label{eq:subset-sum-set}
    \Sigma
    :=
    \left\{
        \sum_{k\in F}4^k:
        F\subset\mathbb N\text{ finite}
    \right\}.
\end{equation}
The set $\Sigma$ has arbitrarily large gaps. This will ultimately be used to contradict the existence of basis projections of every rank.

Next, observe that $\cald$ admits obvious finite rank approximations. Define
\[
    R_N
    =
    I_{\calh_1}\oplus\cdots\oplus I_{\calh_N}
    \oplus0\oplus0\oplus\cdots.
\]
Then
\[
    R_N\in\mathcal D,
    \qquad
    \|R_N\|_{B(\calh)}=1,
\]
and
\[
 \norm{R_N \zeta-\zeta}^2_{\calh} = \sum_{k \ge N+1}\norm{\zeta_k}^2_{\calh_k} \to 0
\]
as $N \to \infty$, i.e., $R_N\zeta \to \zeta$ in $\calh$.

To proceed, we need two lemmas. The first provides a useful bound on the Hilbert-Schmidt norm of traceless operators on a finite-dimensional space. The second shows that our algebra $\cald$ is $\Pi_2$-hyperreflexive.

\begin{lemma}\label{lem:deterministic-half-projection}
Let $A\in M_d(\mathbb C)$ satisfy $\operatorname{Tr}A=0$, where $d$ is even. Then there exists an orthogonal projection $P$ of rank $d/2$ such that
\[
    \|PA(I-P)\|_2^2
    \geq
    \frac{d}{4(d-1)}\|A\|_2^2
    \geq
    \frac14\|A\|_2^2.
\]
\end{lemma}

\begin{proof}
We first claim that there is an orthonormal basis $e_1,\ldots,e_d$ of $\mathbb C^d$ such that
\[
    \langle Ae_j,e_j\rangle=0
    \qquad (1\leq j\leq d).
\]
We prove this by induction on $d$. Since $\operatorname{Tr}A=0$, the origin lies in the convex hull of the eigenvalues of $A$. The numerical range
\[
    W(A):=\{\langle Ax,x\rangle:\|x\|=1\}
\]
contains the spectrum of $A$ and is convex by the Toeplitz--Hausdorff theorem (see e.g.~\cite[1.1-2]{GustafsonRao}). Hence $0\in W(A)$, so there is a unit vector $e_1$ such that $\langle Ae_1,e_1\rangle=0$. Let $A_1$ be the compression of $A$ to $e_1^\perp$. Then
\[
    \operatorname{Tr}A_1
    =\operatorname{Tr}A-\langle Ae_1,e_1\rangle
    =0,
\]
and the induction hypothesis applied to $A_1$ completes the construction of the basis.

Write $A=(a_{ij})$ in this basis. Then $a_{ii}=0$ for every $i$, and therefore
\[
    \|A\|_2^2=\sum_{i\neq j}|a_{ij}|^2.
\]
Set $r=d/2$. For each subset $S\subset\{1,\ldots,d\}$ with $|S|=r$, let $P_S$ denote the orthogonal projection onto
\[
    \operatorname{span}\{e_i:i\in S\}.
\]
Then
\[
    \|P_SA(I-P_S)\|_2^2
    =
    \sum_{\substack{i\in S\\j\notin S}}|a_{ij}|^2.
\]
Summing over all \(r\)-element subsets \(S\), each term \(|a_{ij}|^2\) with \(i\neq j\) occurs exactly $\binom{d-2}{r-1}$ times. As a consequence,
\[
    \sum_{|S|=r}\|P_SA(I-P_S)\|_2^2
    =
    \binom{d-2}{r-1}\|A\|_2^2.
\]
Since there are $\binom dr$ such subsets, at least one of them satisfies
\[
    \|P_SA(I-P_S)\|_2^2
    \geq
    \frac{\binom{d-2}{r-1}}{\binom dr}\|A\|_2^2
    =
    \frac{r(d-r)}{d(d-1)}\|A\|_2^2.
\]
Taking $r=d/2$ gives
\[
    \|P_SA(I-P_S)\|_2^2
    \geq
    \frac{d}{4(d-1)}\|A\|_2^2
    \geq
    \frac14\|A\|_2^2,
\]
as desired.
\end{proof}

\begin{lemma}\label{hyperreflexivelemma}
    Let $\cald$ and $\calh$ be defined as above. Then $\cald$ is $\Pi_2$-hyperreflexive.
\end{lemma}

\noindent \textit{Proof of Lemma \ref{hyperreflexivelemma}.} Let $T \in B(\calh)$ with $\alpha_\cald(T) < \infty$ and let
\[
\delta:= \alpha_\cald(T). 
\]
We will show that there exists $D_T \in \cald$ satisfying 
\[
    \norm{T - D_T}_2 \le 2\sqrt{2}\delta.
\]
To prove this, we note that there are exactly two ways $T \in B(\calh)$ can fail to belong to $\cald$:
\begin{enumerate}
    \item $T$ can have nonzero off-diagonal blocks;
    \item The diagonal blocks $T_{kk}$ can fail to be scalar multiples of the identity.
\end{enumerate}
We will estimate these two cases separately.

We first handle the off-diagonal blocks. Recall that $E_k$ is the projection onto $\calh_k$, and let
\[
    T_{i,j}:=E_iTE_j.
\]
Fix a finite set $F\subset\N$ with $|F|\geq2$. For each subset $S\subset F$, define
\[
    P_S:=\sum_{i\in S}E_i,
    \qquad
    Q_S:=\sum_{j\in F\setminus S}E_j.
\]
Since every $D\in\cald$ is block diagonal, we have $P_SDQ_S=0$. Thus, by the definition of $\delta$,
\[
    \|P_STQ_S\|_2\leq\delta.
\]
The blocks $E_iTE_j$ belong to mutually orthogonal Hilbert--Schmidt subspaces, so
\[
    \|P_STQ_S\|_2^2
    =
    \sum_{\substack{i\in S\\j\in F\setminus S}}
    \|T_{i,j}\|_2^2.
\]
Summing this inequality over all subsets $S\subset F$, each ordered pair $(i,j)$ with $i\neq j$ occurs exactly $2^{|F|-2}$ times. Hence
\[
    2^{|F|-2}
    \sum_{\substack{i,j\in F\\i\neq j}}
    \|T_{i,j}\|_2^2
    =
    \sum_{S\subset F}\|P_STQ_S\|_2^2
    \leq
    2^{|F|}\delta^2.
\]
Therefore
\[
    \sum_{\substack{i,j\in F\\i\neq j}}
    \|T_{i,j}\|_2^2
    \leq4\delta^2.
\]
Since this holds for every finite subset $F\subset\mathbb N$, we obtain
\begin{equation}\label{offdiest}
    \sum_{i\neq j}\|T_{i,j}\|_2^2\leq4\delta^2.
\end{equation}
Thus the off-diagonal part of $T$ is Hilbert--Schmidt and uniformly controlled.

Let us now look at the diagonal blocks. For each $k$, define
\[
    \lambda_k
    :=
    \frac{1}{d_k}\operatorname{Tr}(T_{kk})
\]
and
\[
    A_k
    :=
    T_{kk}-\lambda_kI_{\calh_k}.
\]
Then
\[
    \operatorname{Tr}(A_k) = \operatorname{Tr}(T_{kk}) - \lambda_k\operatorname{Tr}(I_{\calh_k}) = \operatorname{Tr}(T_{kk}) - \frac{1}{d_k}\operatorname{Tr}(T_{kk})d_k=0.
\]
We will prove that
\[
\sum_{k \ge1}\norm{A_k}^2_2 \le 4\delta^2.
\]
Indeed, fix a finite set $F\subset\N$. Since $d_k=4^k$ is even and $\operatorname{Tr}(A_k)=0$, Lemma~\ref{lem:deterministic-half-projection} gives, for each $k\in F$, an orthogonal projection $P_k$ of rank $d_k/2$ on $\calh_k$ such that
\[
    \|P_kA_k(I_{\calh_k}-P_k)\|_2^2
    \geq
    \frac14\|A_k\|_2^2.
\]
Define
\[
    P=\left(\bigoplus_{k\in F}P_k\right)\oplus0\oplus\cdots,
    \qquad
    Q=\left(\bigoplus_{k\in F}(I_{\calh_k}-P_k)\right)\oplus0\oplus\cdots.
\]
Then $P$ and $Q$ are finite-rank contractions. For every $D=\bigoplus_{j \geq 1}\mu_jI_{\calh_j}\in\cald$,
\[
    P_k(\mu_kI_{\calh_k})(I_{\calh_k}-P_k)=0
    \qquad(k\in F),
\]
so $PDQ=0$. Hence
\[
    \|PTQ\|_2\leq\alpha_\cald(T)=\delta.
\]
Moreover, for each $k\in F$,
\[
    (PTQ)_{kk}
    =P_kT_{kk}(I_{\calh_k}-P_k)
    =P_kA_k(I_{\calh_k}-P_k).
\]
Regarding each of these operators on $\calh_k$ as an operator on $\calh$ by extending it by zero on $\calh_k^\perp$, the corresponding diagonal blocks lie in mutually orthogonal Hilbert--Schmidt subspaces. Therefore
\[
    \frac14\sum_{k\in F}\|A_k\|_2^2
    \leq
    \sum_{k\in F}\|P_kA_k(I_{\calh_k}-P_k)\|_2^2
    \leq
    \|PTQ\|_2^2
    \leq\delta^2.
\]
Since the above holds for every finite $F \subset \N$, we conclude that
\begin{equation}\label{eq:diagonal-estimate-guide}
    \sum_{k \ge 1}\|A_k\|_2^2
    \leq4\delta^2,
\end{equation}
as desired.

Now, define
\[
D_T:= \bigoplus_{k \geq 1} \lambda_k
I_{\calh_k}.
\]
Using Cauchy--Schwarz and the estimate $\|T_{kk}\|_2 \leq \sqrt{d_k} \|T_{kk}\|_{B(\calh_k)}$, we obtain
\[
|\lambda_k| = \left|\frac{1}{d_k}\operatorname{Tr}(T_{kk})\right| \le \norm{T_{kk}}_{B(\calh_k)} \le \norm{T}_{B(\calh)} < \infty.
\]
Thus $\{\lambda_k\}_{k \ge1} \in \ell_\infty$ and $D_T \in \cald$.

Consider the operator $T - D_T$, which consists of off-diagonal blocks $T_{i,j}$ and traceless diagonal blocks $A_k$. These pieces are mutually orthogonal in the Hilbert-Schmidt norm, and thus from our estimates (\ref{offdiest}) and (\ref{eq:diagonal-estimate-guide}), we obtain
\[
    \norm{T - D_T}^2_2 = \sum_{i\ne j}\norm{T_{i,j}}^2_2 + \sum_{k \ge 1}\norm{A_k}^2_2 \le 8\delta^2.
\]
Thus, $\norm{T - D_T}_2 \le 2\sqrt{2}\delta$, which shows $\cald$ is indeed $\Pi_2$-hyperreflexive. \qed

\subsection{Applying Oikhberg's Theorem}\label{sec:Oikhberg}

Clearly, our choice of $\calh$ yields a separable Hilbert space. Moreover, $\cald \subset B(\calh)$ is a unital $\Pi_2$-hyperreflexive operator algebra by Lemma \ref{hyperreflexivelemma}. To apply Theorem \ref{thm:oikhberg-specialized}, we need to show that $\ball(\cald)$ is weak$^*$ closed. We know that $\ball(B(\calh))$ is weak$^*$ closed; hence, if we can show that $\cald$ is weak$^*$ closed, this will immediately imply that $\ball(\cald) = \cald \cap \ball(B(\calh))$ is weak$^*$ closed. 

Indeed, let $(D_\alpha)_\alpha \subset \cald$ be a net such that $D_\alpha \overset{w^*}{\to} T$. We need to show that $T \in \cald$. We first show that $T$ has no nonzero off-diagonal elements. Fix $i \ne j$, and let $\zeta\in\calh_j$ and $\eta\in\calh_i$, where we identify $\calh_i$ and $\calh_j$ with their canonical embeddings into $\calh$. Since $D_\alpha$ is block diagonal, it is clear that
\[
\langle D_\alpha \zeta, \eta\rangle_\calh = 0 \qquad \forall \alpha.
\]
In particular, weak$^*$ convergence implies that
\[
0 = \langle D_\alpha\zeta, \eta\rangle_\calh \to \langle T\zeta, \eta\rangle_\calh
\]
so that 
\[
\langle T\zeta, \eta\rangle_\calh = 0.
\]
Since $\zeta, \eta$ were arbitrary, we have that $TE_j\calh \perp E_i\calh$. Thus, $T_{i,j} = E_iTE_j = 0$, and $T$ has no non-zero off-diagonals.

For the diagonal blocks, let us fix a $k \in \N$, fix an orthonormal basis $\zeta_1, \ldots, \zeta_{d_k}$ of $\calh_k$, and assume $D_\alpha = \lambda_k^{(\alpha)} I_{\calh_k}$ on $\calh_k$. For $i, j \in \{1, \ldots, d_k\}$, we have
\[
    \langle D_\alpha \zeta_i, \zeta_j\rangle_\calh = \langle\lambda^{(\alpha)}_k \zeta_i, \zeta_j\rangle_\calh = \begin{cases}
        0 &\text{if } i \ne j\\
        \lambda_k^{(\alpha)} & \text{if } i = j.
    \end{cases}
\]
Taking limits, we obtain
\[
    \langle T\zeta_i, \zeta_j\rangle_\calh = \begin{cases}
        0 &\text{if } i \ne j\\
        \lambda_k & \text{if } i = j.
    \end{cases}
\]
for some scalar $\lambda_k := \lim_\alpha \lambda^{(\alpha)}_k$. But this implies that
\[
\langle T\zeta_i, \zeta_j\rangle_\calh = \langle TE_k\zeta_i, E_k\zeta_j\rangle_\calh = \langle E_kTE_k\zeta_i, \zeta_j\rangle_\calh = \langle T_{kk}\zeta_i, \zeta_j\rangle_\calh =\begin{cases}
        0 &\text{if } i \ne j\\
        \lambda_k & \text{if } i = j.
    \end{cases}
\]
So $T_{kk} = \lambda_kI_{\calh_k}$ and thus,
\[
T = \bigoplus_{k \geq 1}\lambda_kI_{\calh_k},
\]
with
\[
|\lambda_k| = |\langle T\zeta_j, \zeta_j\rangle| \le \norm{T}_{B(\calh)} <\infty,
\]
where $\zeta_j$ is an orthonormal basis vector in $\calh_k$. So $\{\lambda_k\}_{k \ge 1} \in \ell_\infty$ and $T \in \cald$; i.e., $\cald$ is weak$^*$ closed. Consequently, $\ball(\cald)$ is weak$^*$ closed.

Applying Theorem \ref{thm:oikhberg-specialized} to $\calh$ and $\cald$, we obtain an operator-space
structure $X$ on the same underlying Hilbert space $\calh$ and constants
$C_0,C_1>0$ such that:
\begin{enumerate}
    \item if $D\in\mathcal D$ and $\|D\|_{B(\calh)}\leq1$, then
    \begin{equation}\label{eq:block-complete-contraction}
        \|D\|_{\mathrm{cb}}\leq1;
    \end{equation}
    \item if $T\in CB(X)$, then
    \begin{equation}\label{eq:oikhberg-decomposition-guide}
        T=D+S,
        \qquad
        D\in\mathcal D,\quad S\in S_2(\calh),
    \end{equation}
    where
    \begin{equation}\label{eq:oikhberg-decomposition-bounds-guide}
        \|D\|_{B(\calh)}\leq C_0\|T\|_{\mathrm{cb}},
        \qquad
        \|S\|_2\leq C_1\|T\|_{\mathrm{cb}}.
    \end{equation}
\end{enumerate}

With the above in hand, we can now prove Theorem \ref{mainresult}.

\subsection{Proof of Theorem \ref{mainresult}}
The following elementary observation will be useful in the proof.

\begin{lemma}\label{lem:rank-hs-guide}
Let $P$ be a finite-rank idempotent on a Hilbert space and let $Q$ be a
finite-rank orthogonal projection. Then
\[
    |\operatorname{rank}P-\operatorname{rank}Q|
    \leq
    \|P-Q\|_2^2.
\]
\end{lemma}

\begin{proof}
Suppose first that
\[
    \operatorname{rank}P=n>m=\operatorname{rank}Q.
\]
Consider
\[
    Q:\operatorname{Ran}P\longrightarrow\operatorname{Ran}Q.
\]
Its kernel has dimension at least $n-m$. Thus there is an
$(n-m)$-dimensional subspace
\[
    L\subset\operatorname{Ran}P\cap\ker Q.
\]
For $x\in L$,
\[
    Px=x,\qquad Qx=0,
\]
and hence
\[
    (P-Q)x=x.
\]
If $\{\zeta_j\}_{j=1}^{n-m}$ is an orthonormal basis of $L$, then
\[
    \|P-Q\|_2^2
    \geq
    \sum_{j=1}^{n-m}
    \|(P-Q)\zeta_j\|^2
    =
    n-m.
\]
If instead $m>n$, apply the same argument to
\[
    P:\operatorname{Ran}Q\longrightarrow\operatorname{Ran}P.
\]
On the kernel of this map,
\[
    (P-Q)x=-x.
\]
\end{proof}
We can now prove Theorem \ref{mainresult}.

\begin{proof}[Proof of Theorem \ref{mainresult}]
Let $X$ be the operator space structure on $\calh$ obtained in Section \ref{sec:Oikhberg} by applying Theorem \ref{thm:oikhberg-specialized} to $\calh$ and $\cald$. Recall that letting
\[
    R_N
    =
    I_{\calh_1}\oplus\cdots\oplus I_{\calh_N}
    \oplus0\oplus\cdots, 
\]
we have
\[
    R_N\in\mathcal D
    \qquad\text{and}\qquad
    \|R_N\|_{B(\calh)}=1.
\]
Hence \eqref{eq:block-complete-contraction} gives
\[
    \|R_N\|_{\mathrm{cb}}\leq1.
\]
Moreover, each $R_N$ has finite rank, and
\[
    R_N\zeta\to \zeta
    \qquad(\zeta\in \calh).
\]
Therefore $X$ has the 1-CBAP.

We now claim that our space $X$ cannot have a cb-basis. Assume for a contradiction that it does. Let $P_n$ denote its partial-sum projections. Then
\begin{equation}\label{eq:basis-projection-properties-guide}
    P_n^2=P_n,
    \qquad
    \operatorname{rank}P_n=n,
\end{equation}
and there exists $K<\infty$ such that
\begin{equation}\label{eq:uniform-basis-cb-bound-guide}
    \|P_n\|_{\mathrm{cb}}\leq K
    \qquad(n\geq1).
\end{equation}
Apply the Oikhberg decomposition to each $P_n$:
\[
    P_n=D_n+S_n,
\]
where
\[
    D_n\in\mathcal D,
    \qquad
    S_n\in S_2(\calh).
\]
The uniform cb-bound gives, for some $C_0, C_1 > 0$, 
\begin{equation}\label{eq:uniform-D-S-bounds-guide}
    \|D_n\|_{B(\calh)}\leq C_0K,
    \qquad
    \|S_n\|_2\leq C_1K.
\end{equation}
Thus every basis projection is uniformly Hilbert--Schmidt close to a
block-scalar operator. Now, 
\[
    (D_n+S_n)^2=D_n+S_n.
\]
Rearranging,
\begin{equation}\label{eq:Dn-approximate-idempotent-guide}
    D_n^2-D_n
    =
    S_n-D_nS_n-S_nD_n-S_n^2.
\end{equation}
Let
\[
    A:=C_0K,\qquad B:=C_1K.
\]
For $R \in B(\calh)$ and $S\in S_2(\mathcal H)$, using
\[
    \|RS\|_2\leq\|R\|_{B(\calh)}\|S\|_2
\]
and
\[
    \|S^2\|_2
    \leq
    \|S\|_{B(\calh)}\|S\|_2
    \leq
    \|S\|_2^2,
\]
we obtain
\[
    \|D_n^2-D_n\|_2
    \leq
    B+2AB+B^2.
\]
Thus there is a constant $\Theta$, independent of $n$, such that
\begin{equation}\label{eq:theta-guide}
    \|D_n^2-D_n\|_2\leq\Theta.
\end{equation}
Let
\[
    D_n
    =
    \bigoplus_{k \geq 1}
    \lambda_{n,k}I_{\calh_k}.
\]
For each $n,k$, choose
\[
    \varepsilon_{n,k}\in\{0,1\}
\]
to be whichever of $0$ and $1$ is closer to $\lambda_{n,k}$. If $|\lambda_{n,k}| = |\lambda_{n,k} -1|$, then choose $\varepsilon_{n,k} = 0$. Define
\[
    Q_n
    :=
    \bigoplus_k
    \varepsilon_{n,k}I_{\calh_k}.
\]
Then $Q_n$ is an orthogonal block projection.
Since the square of the Hilbert-Schmidt norm of a diagonal operator is the $\ell^2$-sum of its diagonal entries and since $\norm{I_{\calh_k}}^2_2 = \operatorname{Tr}(I_{\calh_k})= d_k$, we observe that
\begin{align*}
    \|D_n-Q_n\|_2^2
    &=\sum_{k \ge 1}
    d_k|\lambda_{n,k}-\varepsilon_{n,k}|^2\\
    &\leq
    4
    \sum_{k \ge 1}
    d_k|\lambda_{n,k}(\lambda_{n,k}-1)|^2\\
    &=
    4\|D_n^2-D_n\|_2^2,
\end{align*}
where the inequality comes from the fact that
\[
|\lambda_{n,k} - \varepsilon_{n,k}| = \min\{|\lambda_{n,k}|, |\lambda_{n,k} -1|\},
\]
and  
\[
\min\{|z|,|z-1|\} \leq 2|z(z-1)|
\]
for all $z \in \C$. This inequality can easily be seen from the fact that
\[
    \max\{|z|, |z-1|\} \ge 1/2
\]
for all $z \in \C$. Thus,
\begin{equation}\label{eq:Dn-Qn-guide}
    \|D_n-Q_n\|_2\le
    2\norm{D_n^2 -D_n}_2\leq2\Theta.
\end{equation}
As a consequence,
\[
    \|P_n-Q_n\|_2
    \leq
    \|S_n\|_2+\|D_n-Q_n\|_2
    \leq
    B+2\Theta.
\]
Let $R:= B + 2\Theta$. Then
\begin{equation}\label{eq:Pn-Qn-guide}
    \|P_n-Q_n\|_2\leq R
    \qquad\text{for all }n.
\end{equation}
This shows that a uniformly cb-family of idempotents is forced to remain within a fixed Hilbert--Schmidt distance of the family of block projections.

We now claim that the operators $Q_n$ have finite rank. First observe that since $P_n$ has finite rank and $S_n\in S_2(\calh)$, we must have $D_n = P_n - S_n \in S_2(\calh)$. Now suppose that $\varepsilon_{n,k} = 1$. Then
\[
    |\lambda_{n,k} - 1| \le |\lambda_{n,k}|
\]
by definition, so that
\[
|\lambda_{n,k}| \ge 1/2.
\]
We know that
\[
\norm{D_n}^2_{2} = \sum_{k \ge 1}d_k|\lambda_{n,k}|^2 < \infty.
\]
In particular, there can only be finitely many $\varepsilon_{n,k}$ which satisfy $\varepsilon_{n,k} = 1$ and so $Q_{n}$ has finite rank. More specifically,
\begin{equation}\label{eq:rank-Qn-subset-sum-guide}
\operatorname{rank}Q_n
    =
    \sum_{k\in F_n}4^k
    \in\Sigma
\end{equation}
for some finite set $F_n \subset \N$.

Now, apply Lemma~\ref{lem:rank-hs-guide} to $P_n$ and $Q_n$. From
\eqref{eq:basis-projection-properties-guide} and
\eqref{eq:Pn-Qn-guide},
\[
    |n-\operatorname{rank}Q_n|
    \leq R^2.
\]
Using \eqref{eq:rank-Qn-subset-sum-guide}, we conclude that
\begin{equation}\label{eq:distance-to-Sigma-guide}
    \operatorname{dist}(n,\Sigma)\leq R^2
    \qquad\text{for every }n\geq1.
\end{equation}
However, observe that the largest subset sum using only
\[
    4,4^2,\ldots,4^k
\]
is
\[
    4+4^2+\cdots+4^k
    =
    \frac{4^{k+1}-4}{3}.
\]
Any subset sum which uses the next block has size at least $4^{k+1}$. Therefore the interval
\begin{equation}\label{eq:subset-sum-gap-guide}
    \left(
        \frac{4^{k+1}-4}{3},
        4^{k+1}
    \right)
    \cap\Sigma
    =
    \varnothing.
\end{equation}
The length of this interval is
\[
    4^{k+1}
    -
    \frac{4^{k+1}-4}{3}
    =
    \frac{2\cdot4^{k+1}+4}{3} \to \infty
\]
as $k \to \infty$. Hence
\[
    \sup_{n\geq1}\operatorname{dist}(n,\Sigma)=\infty.
\]
This contradicts \eqref{eq:distance-to-Sigma-guide}, and shows our operator space $X$ cannot have a cb-basis. This concludes the proof of Theorem \ref{mainresult}.
\end{proof}
\bigskip

\noindent{\bf Acknowledgments}

D.G. was partially supported by NSF grant \#2350067.

\bigskip

\noindent{\bf AI disclosure statement.} The authors acknowledge the use of ChatGPT by OpenAI to assist in exploring and refining some of the arguments presented in this work. The authors take full responsibility for the content of the paper.


\begin{thebibliography}{99}

\bibitem{CasazzaHanLarson1999}
P.~G.~Casazza, D.~Han, and D.~R.~Larson,
\emph{Frames for Banach spaces},
in \emph{The Functional and Harmonic Analysis of Wavelets and Frames},
Contemp. Math., vol.~247,
Amer. Math. Soc., Providence, RI, 1999, 149--182.

\bibitem{DiestelJarchowTonge1995}
J.~Diestel, H.~Jarchow, and A.~Tonge,
\emph{Absolutely Summing Operators},
Cambridge Studies in Advanced Mathematics, vol.~43,
Cambridge University Press, Cambridge, 1995.

\bibitem{EffrosRuan1990}
E.~G.~Effros and Z.-J.~Ruan,
\emph{On approximation properties for operator spaces},
Internat. J. Math. \textbf{1} (1990), 163--187.

\bibitem{EffrosRuan2000}
E.~G.~Effros and Z.-J.~Ruan,
\emph{Operator Spaces},
London Mathematical Society Monographs, New Series, vol.~23,
Oxford University Press, Oxford, 2000.

\bibitem{GustafsonRao}
K.~E.~Gustafson and D.~K.~M.~Rao,
\emph{Numerical Range: The Field of Values of Linear Operators and Matrices},
Universitext, Springer-Verlag, New York, 1997.

\bibitem{JungeNielsenRuanXu2004}
M.~Junge, N.~J.~Nielsen, Z.-J.~Ruan, and Q.~Xu,
\emph{$COL_p$ spaces---the local structure of non-commutative
$L_p$ spaces},
Adv. Math. \textbf{187} (2004), 257--319.

\bibitem{LiuRuan2016}
R.~Liu and Z.-J.~Ruan,
\emph{Cb-frames for operator spaces},
J. Funct. Anal. \textbf{270} (2016), 4280--4296.

\bibitem{Oikhberg2005}
T.~Oikhberg,
\emph{Operator spaces with prescribed sets of completely bounded maps},
J. Funct. Anal. \textbf{224} (2005), 296--315.

\bibitem{Oikhberg2007}
T.~Oikhberg,
\emph{Hyperreflexivity and operator ideals},
J. Funct. Anal. \textbf{246} (2007), 242--280.

\bibitem{Oikhberg2009}
T.~Oikhberg,
\emph{Representations of Banach algebras as algebras of completely
bounded maps},
Math. Scand. \textbf{105} (2009), 99--120.

\bibitem{OikhbergRicard2004}
T.~Oikhberg and \'E.~Ricard,
\emph{Operator spaces with few completely bounded maps},
Math. Ann. \textbf{328} (2004), 229--259.

\bibitem{Paulsen2002}
V.~Paulsen,
\emph{Completely Bounded Maps and Operator Algebras},
Cambridge Studies in Advanced Mathematics, vol.~78,
Cambridge University Press, Cambridge, 2002.

\bibitem{Pisier2003}
G.~Pisier,
\emph{Introduction to Operator Space Theory},
London Mathematical Society Lecture Note Series, vol.~294,
Cambridge University Press, Cambridge, 2003.


\bibitem{Szarek1987}
S.~J.~Szarek,
\emph{A Banach space without a basis which has the bounded approximation property},
Acta Math. \textbf{159} (1987), 81--98.
\end{thebibliography}
\end{document}